\documentclass[11pt,twoside]{amsart}    

\usepackage{prefix}
\providecommand{\bstar}{\mathbin{\bar{\star}}}

\title{Shifted lower Bruhat intervals are EL-shellable}
 \author{Aman Dave}
 \address{Westwood High School} 
 \email{amanpdave@gmail.com}
 \author{Taiwen Feng}
 \address{Cranbrook Kingswood Upper School} 
 \email{tfeng866@gmail.com}
  \author{Nathan Lesnevich}
 \address{Oklahoma State University} 
  \email{nlesnev@okstate.edu}
 \author{SuHo Oh}
 \address{Texas State University} 
 \email{suhooh@txstate.edu}
 \author{Edward Richmond}
 \address{Oklahoma State University} 
\email{edward.richmond@okstate.edu}
 \author{Ethan Wang}
 \address{West Windsor-Plainsboro High School North} 
 \email{ethan.wang2918@gmail.com}
 \author{Aaron Wu}
 \address{Seven Lakes High School} 
 \email{aaron.wu.401@gmail.com}


\begin{document}

\begin{abstract}
Let $W$ be an arbitrary Coxeter group. The shifted Bruhat interval
$[w_1,w_2]\,x^{-1}$, the translate of the Bruhat interval $[w_1,w_2]$ by an element $x$, is partially ordered by the Bruhat order of $W$. These posets arise from affine pavings of Richardson varieties, and in general they are neither twisted intervals nor tilted Bruhat intervals. Our main result is that the shifted lower intervals $\swx$ are EL-shellable for every Coxeter group, via an explicit labeling of each cover by a reflection. Along the way we show that $\swx$ is a graded poset with a unique maximum given by the Demazure product and a unique minimum given by an opposite Demazure operator that we introduce.
\end{abstract} 
\maketitle

\section{Introduction}
The Bruhat order of a Coxeter group $(W,S)$ is one of the central objects of algebraic combinatorics, and its topology is remarkably well behaved in that every interval is shellable. Highlights of this story include Verma's computation of the M\"obius function \cite{Verma}, the lexicographic shellability framework of Bj\"orner and Wachs \cite{BW82,BW83}, and the explicit EL-labelings constructed by Edelman and Proctor in the classical types \cite{Edelman,Proctor}. Dyer \cite{Dyer93} showed that any reflection ordering of the reflections of $W$ yields a uniform EL-labeling of every Bruhat interval, obtained by labeling each cover relation with its associated reflection.
Consequently, Bruhat intervals are shellable and Cohen-Macaulay, and the order complex of each open interval is homeomorphic to a sphere \cite{BW82}.

In this paper, we study a family of posets obtained by shifting a Bruhat interval by a fixed element. For $w, x \in W$ define the \newword{shifted (lower) Bruhat interval}
\[
  \swx  \;=\; \{\, vx^{-1} : v \le w \,\},
\]
regarded as a poset under the Bruhat order of $W$ restricted to this set.

An alternative view keeps the elements of $[e,w]$ and changes the order through the bijection $v \mapsto vx^{-1}$.  In particular, for $u,v\in [e,w]$, one could equip the pulled-back order 
\[u \leq_x v \iff ux^{-1} \le vx^{-1}.\]
By the cocycle identity for inversion sets \cite[\S 2]{Dyer90}, this
pulled-back order is Dyer's \newword{twisted Bruhat order} $\le_A$ where the twisting set $A$ is given by the set of left inversions of $x^{-1}$ \cite{Dyer92}. Thus $\swx$ is fundamentally a hybrid combinatorial object: its ground set is carved out by the untwisted Bruhat order, whereas its relations are governed by the twisted one. As we explain below, this hybridity is exactly what places $\swx$ outside the reach of the existing shellability theorems, and resolving it is the main content of this paper.


Our primary motivation comes from algebraic geometry. In an upcoming project \cite{LOR}, we study affine pavings of Richardson varieties obtained via intersections with the Schubert cells of a \newword{shifted Borel subgroup}. The cells constituting these pavings are naturally indexed by shifted Bruhat intervals, and the combinatorics of their induced order dictates the
cell-attaching data.

To contextualize shifted intervals, it is instructive to contrast them with two established generalizations of the Bruhat order:
\begin{itemize}
\item \newword{Twisted Bruhat orders} ($\le_A$): Introduced by Dyer \cite{Dyer92,Dyer94} for a biclosed set of reflections $A$, these posets are shelled by reflection orderings featuring $A$ as an initial section. However, for a finite group $W$, every biclosed set is an inversion set, rendering $\le_A$ merely a right translate of the standard Bruhat order. Thus, twisted intervals are ordinary Bruhat intervals in disguise.
\item \newword{Tilted Bruhat orders}: Defined by Brenti, Fomin, and Postnikov \cite{BFP} using the quantum Bruhat graph, these orders deform the Bruhat order along a different axis by
incorporating the length-decreasing \newword{quantum edges}. Tilted intervals are lexicographically shellable and Eulerian: every tilted interval is the face poset of a shellable regular CW sphere \cite{BFP} (see \cite{GGG} for their connections with tilted Richardson varieties, their geometric counterparts).
\end{itemize}

Shifted intervals belong to neither family. As can be checked from Figure~\ref{fig:longstemex} and Example~\ref{ex:longstem}, they routinely violate thinness: they can contain rank-$2$ intervals that are chains. Hence they are in general neither Eulerian \cite{Stanley12} nor isomorphic to tilted or twisted Bruhat intervals.

Despite this structural irregularity, shifted lower
intervals retain remarkably rigid combinatorial properties.  First, we show they have a unique minimum and maximum.  The maximum arises from the \newword{Demazure product}, and the minimum from an \newword{opposite Demazure operator} that we introduce in Section 4. Next, we show that they are graded posets. This is not automatic, as many shifted intervals fail to be graded (for example, see Figure~\ref{fig:nongraded}). Lastly, we have the following main theorem on shellability:



\begin{theorem}\label{thm:main_intro}
Let $W$ be an arbitrary Coxeter group and $w, x \in W$. The shifted lower Bruhat interval $\swx$ is EL-shellable.
\end{theorem}

We emphasize that Theorem \ref{thm:main_intro} holds for all Coxeter groups.  For a precise statement on the corresponding reflection ordering and EL-labeling of $\swx$, see Theorem \ref{thm:main}.

This paper is organized as follows. In Section \ref{section:prelim}, we collect the
necessary preliminaries on Coxeter groups, reflection orderings, and shellability. In Section \ref{section:shift_interval}, we introduce shifted intervals, illustrate them in a running example, and establish the ``switching lemma" (see Lemma \ref{lem:switch}) that drives the rest of the paper. In Section \ref{section:shift_lower_intervals}, we prove the structural properties of shifted lower intervals: unique extrema via the Demazure product and the opposite Demazure operator, and gradedness via the weak generalized lifting property of Caselli, D'Adderio, and Marietti. Section \ref{section:EL_labels}
contains the proof of our main theorem, decomposing the argument into a general poset criterion and a localized broken-diamond analysis. We conclude in Section \ref{section:conclusions} with open questions regarding middle intervals.

\section{Preliminaries}\label{section:prelim}
Throughout this paper, $(W,S)$ is a Coxeter system, not assumed finite, with identity $e$, length function $\ell:W\rightarrow \ZZ_{\geq 0}$, and set of reflections $T := \bigcup_{w \in W} wSw^{-1}$. When $W$ is finite we write $w_0$ for its longest element. We write elements (permutations) of the symmetric group $S_n$ (type $A_{n-1}$) in one-line notation, often as strings of integers, e.g.\ $3412$, and write $t_{ab} \in T$ for the transposition exchanging the values $1\leq a<b\leq n$. We refer to \cite{Bjorner-Brenti05} for standard background on Coxeter groups.

The \newword{Bruhat order} $\le$ on $W$ is defined by declaring $u \le v$ if some reduced word of $u$ appears as a subword of some (equivalently, every) reduced word of $v$ \cite[Theorem~2.2.2]{Bjorner-Brenti05}.  For $w_1\le w_2$, we use the standard notation 
\[ [w_1,w_2]:=\{u\in W : w_1\leq u\leq w_2\}\] 
to denote the corresponding interval. Bruhat order is graded by length and the cover relations are exactly the pairs $u \lessdot v$ with $v = ut$ for some $t \in T$ and $\ell(v) = \ell(u)+1$. Each cover carries two natural reflections and we fix names for both.  For any cover relation $u \lessdot v$, define the \newword{right label} $\lambda_R(u \lessdot v) := u^{-1}v\in T$ and the \newword{left label} $\lambda_L(u \lessdot v) := vu^{-1}\in T$. The distinction between these two labels will be essential in this paper. Recall that the map $w \mapsto w^{-1}$ is an automorphism of Bruhat order and, when $W$ is finite, the map $w \mapsto ww_0$ is an anti-automorphism \cite[Proposition~2.3.4]{Bjorner-Brenti05}.

The \newword{(left) inversion set} of $w \in W$ is
\[
  N(w) \;:=\; \{\, t \in T : \ell(tw) < \ell(w) \,\}.
\]
Note that $|N(w)| = \ell(w)$ and $w$ is determined by $N(w)$. For the symmetric group $S_n$, we have
$N(w) = \{ t_{ab} : a < b,\; w^{-1}(a) > w^{-1}(b) \}$.

The standard geometric representation of $(W,S)$ on a real vector space $V$ comes with a root system of positive and negative roots $\Phi = \Phi^+ \sqcup \Phi^-$, simple roots $\{\alpha_s : s \in S\}$, and a bijection from $T$ to $\Phi^+$ given by $t \mapsto \alpha_t$ \cite[Chapter~4]{Bjorner-Brenti05}. For $w \in W$ and $t \in T$ we have $\ell(tw) < \ell(w)$ if and only if $w^{-1}(\alpha_t) \in \Phi^-$.  So $t \in N(x)$ if and only if $x^{-1}(\alpha_t)$ is a negative root. The simple roots are linearly independent, and every positive root is a nonnegative linear combination of simple roots.  In particular, any linear functional that is strictly positive on the simple roots is strictly positive on $\Phi^+$ and strictly negative on $\Phi^-$.

Let $T_R(w) = N(w^{-1})$ denote the \newword{right inversion set} of $w$.  We get the following \newword{cocycle identity} where $\Delta$ stands for the symmetric difference \cite[Chapter~1, Exercise~12]{Bjorner-Brenti05}, \cite[\S 2]{Dyer90}:
\begin{equation}\label{eqn:cocycleR}
  T_R(uv) \;=\; T_R(v) \,\Delta\, v^{-1} T_R(u) v .
\end{equation}

We will use repeatedly that Bruhat intervals are \newword{thin}, in that every
interval $[w_1,w_2]$ with $\ell(w_2) - \ell(w_1) = 2$ contains exactly two elements strictly between $w_1$ and $w_2$ \cite[Lemma~2.7.3]{Bjorner-Brenti05}.

\subsection{Reflection orderings}

In \cite[Theorem 3.3]{Dyer90}, Dyer proves that if $W'$ is a subgroup of $W$ generated by a subset of reflections in $T$, then $W'$ is a Coxeter group with generating set $\chi(W'):=\{t\in T : N(t)\cap W'=\{t\}\}.$  Any such group is called a \newword{reflection subgroup} of $W$.
 Later in  \cite{Dyer93}, Dyer defines a certain total order on $T$ as follows using rank 2 reflection subgroups (see also \cite[\S 5.1]{Bjorner-Brenti05}).  A total order $\prec$ on $T$ is a \newword{reflection ordering} if its restriction to every dihedral reflection subgroup $W'\subseteq W$ with generators $\chi(W')=\{r,t\}$ is one of the two \newword{string orders}:
\[
  r \prec rtr \prec rtrtr \prec \cdots \prec trt \prec t
  \quad \text{or its reverse.}
\]
Reflection orderings exist for every Coxeter system, even in the infinite case \cite{Dyer93}. For the symmetric group $S_n$, the condition says: for all $1\leq a < b < c\leq n$, the transposition $t_{ac}$ lies $\prec$-between $t_{ab}$ and $t_{bc}$. For example, 
\begin{equation}\label{eqn:S4_reflection_order_example}
t_{12} \prec t_{13} \prec t_{23} \prec t_{14} \prec t_{24} \prec t_{34}
\end{equation}
is a reflection ordering of the transpositions of $S_4$.

When $W$ is finite, one way to select a reflection ordering is to choose a reduced word of $w_0 = s_1\dots s_N$ and define the ordering 
\[s_1 \prec s_1s_2s_1 \prec \dots \prec s_1\dots s_{N-1}s_Ns_{N-1}\dots s_1.
\]
By \cite[Chapter 5, Exercise 20]{Bjorner-Brenti05}, the above ordering is a reflection ordering.

An \newword{initial section} of a reflection ordering $\prec$ is a set of the form $\{ t : t \prec t_0 \}$ for some $t_0\in T$.  We also allow the sets $\emptyset$ and $T$ to be initial sections. In \cite[Lemma 2.11]{Dyer93}, Dyer proves the following characterization of (finite) initial sections for arbitrary Coxeter groups.

\begin{lemma}\cite[Lemma 2.11]{Dyer93}
\label{prop:biclosed}
Let $A \subseteq T$ be a finite set. Then $A = N(x)$ for some $x \in W$ if and only if $A$ is an initial section of some reflection ordering.
\end{lemma}

We call a reflection ordering having $N(x)$ as an initial section \newword{$N(x)$-initial}. By Lemma~\ref{prop:biclosed}, $N(x)$-initial orderings exist for every $x$ in every Coxeter group. When $W$ is finite they are easy to produce explicitly. Given a reduced word $w_0 = s_1 s_2 \cdots s_N$, the sequence of prefix conjugates
\[
  t_k \;=\; s_1 s_2 \cdots s_{k-1}\, s_k\, s_{k-1} \cdots s_2 s_1,
  \qquad 1 \le k \le N,
\]
lists $T$ without repetition.  The ordering $t_1 \prec t_2 \prec \cdots \prec t_N$ is a reflection ordering whose initial sections are exactly the sets $N(s_1 \cdots s_k)$. Consequently, an $N(x)$-initial reflection ordering can be read off any reduced word of $w_0$ beginning with a reduced word of $x$. The ordering of $S_4$ displayed in Equation \eqref{eqn:S4_reflection_order_example} is $N(x)$-initial for $x = 3214$ with associated reduced word $w_0 = (s_1 s_2 s_1) s_3 s_2 s_1$.

\subsection{EL-labelings}
Let $P$ be a finite graded poset with $\hat0$ and $\hat1$, and let
$\lambda$ be a map from the cover relations of $P$ to a totally ordered set $(\Lambda, \prec)$. A maximal chain
$p = c_0 \lessdot c_1 \lessdot \cdots \lessdot c_k = q$ of an interval
$[p,q]$ is \newword{increasing} if $\lambda(c_0 \lessdot c_1) \prec \lambda(c_1 \lessdot c_2) \prec \cdots \prec \lambda(c_{k-1} \lessdot c_k)$, and maximal chains of $[p,q]$ are compared in the lexicographic order on their label sequences. 

\begin{definition}[\cite{BW82,BW83}]
\label{def:ELlabeing}
The labeling $\lambda$ is an \newword{EL-labeling} if the following conditions are satisfied:
\begin{enumerate}
    \item every interval $[p,q]$ of $P$ has a
unique increasing maximal chain, and
\item this chain is lexicographically first among the maximal chains of $[p,q]$.
\end{enumerate}
\end{definition}
The \newword{order complex} of a poset is the simplicial complex whose faces are its chains. We call a bounded poset \newword{shellable} (respectively \newword{Cohen-Macaulay}) if the order complex of its proper part is shellable (respectively Cohen-Macaulay). A bounded graded poset admitting an EL-labeling is called \newword{EL-shellable}. In particular, EL-shellable posets are shellable \cite{BW82,BW83} and hence Cohen-Macaulay \cite{Stanley96}.

The key input to our proofs is Dyer's uniform EL-labeling of Bruhat intervals, which is valid for every Coxeter system.
 
\begin{theorem}[{\cite{Dyer93}}]\label{thm:dyer}
Let $\prec$ be any reflection ordering of $T$. Then the left labeling
$\lambda_L(u \lessdot v) = vu^{-1}$ is an EL-labeling of every interval of
$(W, \le)$.
\end{theorem}

Theorem~\ref{thm:dyer} holds for the right labeling
$\lambda_R$ as well. Indeed, $u \mapsto u^{-1}$ is an automorphism of the Bruhat order carrying $[w_1,w_2]$ to $[w_1^{-1}, w_2^{-1}]$, and it converts right labels into left labels. To see this, note that if $u \lessdot v$, then
$\lambda_L(u^{-1} \lessdot v^{-1}) = v^{-1}u = (u^{-1}v)^{-1} =
\lambda_R(u \lessdot v)$, since reflections are involutions. In particular, for each reflection ordering the two labelings of a fixed interval have the same increasing and lexicographically-first chains after the transformation. In this paper we are going to stick to right labelings.

\section{Shifted intervals}\label{section:shift_interval}
In this section, we introduce our main object of study and give several examples of shifted intervals. Although our main result concerns only shifted lower intervals, many of the tools we develop apply to arbitrary shifted intervals.

\begin{definition}\label{def:shifted}
For $w_1,w_2,x \in W$ with $w_1 \leq w_2$, the \newword{shifted Bruhat interval} is
\[
  \sywx \;=\; \{\, ux^{-1} : w_1 \leq u \leq w_2 \,\},
\]
partially ordered by the restriction of the Bruhat order of $W$. We refer to the bijection $u \mapsto ux^{-1}$ from $[w_1,w_2]$ to $\sywx$ as the \newword{shift by $x$}. When $w_1=e$, we call $\swx$ a \newword{shifted lower interval}.
\end{definition}

The shift interacts well with reflections: if $u$ and $ut$ differ by the reflection $t \in T$, then their images $ux^{-1}$ and $(ut)x^{-1} = ux^{-1}r$ differ by the reflection $r = xtx^{-1}$, and any
two elements of $W$ differing by a reflection are Bruhat-comparable. So pairs differing by a reflection remain comparable after the shift, even though comparability of general pairs may be lost. Note that the direction of the comparison may reverse, and the length gap may change. 

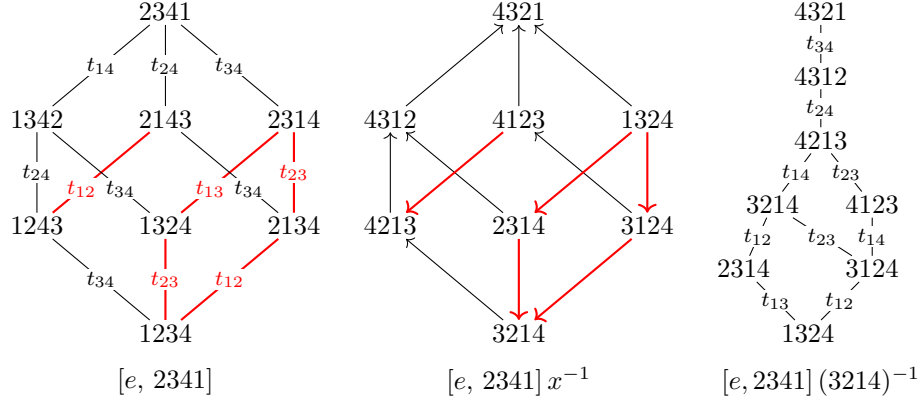
\begin{figure}[h]
 \begin{center}
  \vspace{3.7mm}
\begin{tikzpicture}[scale=0.85,
  every node/.style={font=\small},
  elt/.style={inner sep=1.5pt},
  lab/.style={font=\scriptsize, fill=white, inner sep=0.8pt},
  flip/.style={red, thick},
  fliplab/.style={font=\scriptsize, fill=white, inner sep=0.8pt, text=red}]
 \begin{scope}
  \node[elt] (b1234) at (0,0)        {$1234$};
  \node[elt] (b1243) at (-2,1.67)    {$1243$};
  \node[elt] (b1324) at (0,1.67)     {$1324$};
  \node[elt] (b2134) at (2,1.67)     {$2134$};
  \node[elt] (b1342) at (-2,3.33)    {$1342$};
  \node[elt] (b2143) at (0,3.33)     {$2143$};
  \node[elt] (b2314) at (2,3.33)     {$2314$};
  \node[elt] (b2341) at (0,5)        {$2341$};
  \draw        (b1234) -- node[lab,pos=.5]     {$t_{34}$} (b1243);
  \draw[flip]  (b1234) -- node[fliplab,pos=.5] {$t_{23}$} (b1324);
  \draw[flip]  (b1234) -- node[fliplab,pos=.5] {$t_{12}$} (b2134);
  \draw        (b1243) -- node[lab,pos=.5]     {$t_{24}$} (b1342);
  \draw[flip]  (b1243) -- node[fliplab,pos=.3] {$t_{12}$} (b2143);
  \draw        (b1324) -- node[lab,pos=.3]     {$t_{34}$} (b1342);
  \draw[flip]  (b1324) -- node[fliplab,pos=.3] {$t_{13}$} (b2314);
  \draw        (b2134) -- node[lab,pos=.3]     {$t_{34}$} (b2143);
  \draw[flip]  (b2134) -- node[fliplab,pos=.5] {$t_{23}$} (b2314);
  \draw        (b1342) -- node[lab,pos=.5]     {$t_{14}$} (b2341);
  \draw        (b2143) -- node[lab,pos=.5]     {$t_{24}$} (b2341);
  \draw        (b2314) -- node[lab,pos=.5]     {$t_{34}$} (b2341);
  \node at (0,-0.75) {$[e,\,2341]$};
 \end{scope}
 \begin{scope}[shift={(5.5,0)}]
  \node[elt] (m3214) at (0,0)        {$3214$};
  \node[elt] (m4213) at (-2,1.67)    {$4213$};
  \node[elt] (m2314) at (0,1.67)     {$2314$};
  \node[elt] (m3124) at (2,1.67)     {$3124$};
  \node[elt] (m4312) at (-2,3.33)    {$4312$};
  \node[elt] (m4123) at (0,3.33)     {$4123$};
  \node[elt] (m1324) at (2,3.33)     {$1324$};
  \node[elt] (m4321) at (0,5)        {$4321$};
  \draw[->]        (m3214) -- (m4213);
  \draw[->]        (m4213) -- (m4312);
  \draw[->]        (m2314) -- (m4312);
  \draw[->]        (m4312) -- (m4321);
  \draw[->]        (m3124) -- (m4123);
  \draw[->]        (m4123) -- (m4321);
  \draw[->]        (m1324) -- (m4321);
  \draw[flip,->]   (m2314) -- (m3214);
  \draw[flip,->]   (m3124) -- (m3214);
  \draw[flip,->]   (m4123) -- (m4213);
  \draw[flip,->]   (m1324) -- (m2314);
  \draw[flip,->]   (m1324) -- (m3124);
  \node at (0,-0.75) {$[e,\,2341]\,x^{-1}$};
 \end{scope}
 \begin{scope}[shift={(10,0)}]
  \node[elt] (a1324) at (0,0)     {$1324$};
  \node[elt] (a2314) at (-1,1)    {$2314$};
  \node[elt] (a3124) at (1,1)     {$3124$};
  \node[elt] (a3214) at (-0.55,2) {$3214$};
  \node[elt] (a4123) at (1,2)     {$4123$};
  \node[elt] (a4213) at (0.2,3)   {$4213$};
  \node[elt] (a4312) at (0.2,4)   {$4312$};
  \node[elt] (a4321) at (0.2,5)   {$4321$};
  \draw (a1324) -- node[lab,pos=.5] {$t_{13}$} (a2314);
  \draw (a1324) -- node[lab,pos=.5] {$t_{12}$} (a3124);
  \draw (a2314) -- node[lab,pos=.5] {$t_{12}$} (a3214);
  \draw (a3124) -- node[lab,pos=.5] {$t_{23}$} (a3214);
  \draw (a3124) -- node[lab,pos=.5] {$t_{14}$} (a4123);
  \draw (a3214) -- node[lab,pos=.5] {$t_{14}$} (a4213);
  \draw (a4123) -- node[lab,pos=.5] {$t_{23}$} (a4213);
  \draw (a4213) -- node[lab,pos=.5] {$t_{24}$} (a4312);
  \draw (a4312) -- node[lab,pos=.5] {$t_{34}$} (a4321);
  \node at (0.2,-0.75) {$[e,2341]\,(3214)^{-1}$};
 \end{scope}
\end{tikzpicture}
    \captionsetup{width=1.0\linewidth}
  \captionof{figure}{The interval $[e,2341]$ with flipping edges in red
  (left), the shift of its elements by $x = 3214$ with edges directed
  toward the Bruhat-larger element (middle), and the Hasse diagram of
  $[e,2341]\,(3214)^{-1}$ (right). See Example~\ref{ex:longstem}.}
  \label{fig:longstemex}
 \end{center}
\end{figure}

\begin{example}\label{ex:longstem}
Let $W = S_4$, $w = 2341$ and $x = 3214$, so that $N(x) = \{t_{12}, t_{13}, t_{23}\}$. The left and middle panel of Figure~\ref{fig:longstemex} show the passage from $[e, w]$ to $\swx$.

The left panel is the ordinary interval $[e, 2341]$, with each cover
$u \lessdot ut$ marked by its right label $t$. An edge is drawn in red
when the shift by $x$ reverses its Bruhat comparison. 

The middle panel
applies the shift where each vertex $u$ is replaced by $ux^{-1}$, keeping the
layout.  The direction of each edge is adjusted to point from the Bruhat-smaller to the
Bruhat-larger element. Black edges keep their orientation while red edges
reverse. 

The length gaps can also change: the pair $2314$ and $2341$ in the original interval had a length difference of $1$, their shifts $1324$ and $4321$ have length difference of $5$.
\end{example}

To picture a shifted interval we will draw the Hasse diagram of $\sywx$
(or $\swx$) alongside that of the original interval $[w_1,w_2]$: the shift
preserves the vertex set up to relabeling, but it reorients some edges and stretches others, so the two Hasse diagrams can look drastically
different. 

The right panel of Figure~\ref{fig:longstemex} is the Hasse diagram of $[e,2341]\,(3214)^{-1}$, drawn by the usual length of the Bruhat order: the three stretched comparisons are no longer cover relations in the induced order, so the twelve edges of the middle panel become nine covers.

Example \ref{ex:longstem} shows how far shifting can change the poset structure of a
Bruhat interval. The gaps in length remain odd (two elements differing by a
reflection have lengths of opposite parity), but covers need not survive, and the resulting poset looks nothing like a standard Bruhat interval.  

For another example, the rank-$2$ interval $[4213, 4321]$ of $[e,2341]\,(3214)^{-1}$ is a chain, violating thinness, and the whole poset is an interval of rank $5$ with a unique coatom. Consequently $[e,2341]\,(3214)^{-1}$ is not isomorphic to any Bruhat interval.  For finite Coxeter groups, every twisted order is a right translate of the Bruhat order and hence $[e,2341]\,(3214)^{-1}$ is not an interval of a twisted order either.  Moreover since every interval of a tilted Bruhat order is Eulerian, hence thin \cite[Corollary~6.5]{BFP}, it is not isomorphic to a tilted Bruhat interval of any Weyl group.

Even the global shape of the rank sizes can be disturbed, as the next example shows.

\begin{example}\label{ex:bowtie}
Let $W = S_4$ and $w = x = 2413$, so that $N(x) = \{t_{12}, t_{14}, t_{34}\}$. The interval $[e,2413]$ is isomorphic to the Boolean lattice on three elements. Figure~\ref{fig:bowtie} shows the passage to $[e,2413]\,(2413)^{-1}$: eight of the twelve covers flip, and the resulting poset is graded of rank $5$ with rank sequence $1,2,1,1,2,1$ from bottom to top. In particular, the rank generating function of a shifted lower interval need not be unimodal.
\end{example}

Lower Bruhat intervals behave very differently: Bj\"orner and Ekedahl \cite{BE09} showed that the rank numbers of $[e,w]$ in a crystallographic Coxeter group satisfy $f_i \le f_j$ whenever $0 \le i < j \le \ell(w) - i$, so in particular they increase along the bottom half. The shift can destroy even this coarse constraint on the shape.

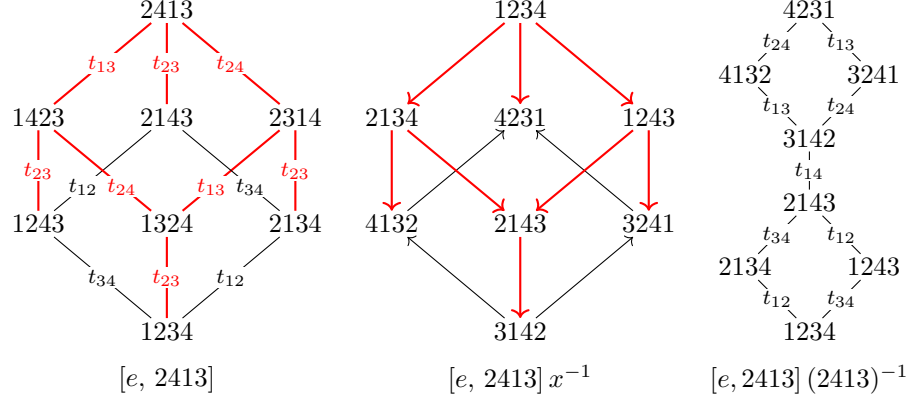
\begin{figure}[h]
 \begin{center}
  \vspace{3.7mm}
\begin{tikzpicture}[scale=0.85,
  every node/.style={font=\small},
  elt/.style={inner sep=1.5pt},
  lab/.style={font=\scriptsize, fill=white, inner sep=0.8pt},
  flip/.style={red, thick},
  fliplab/.style={font=\scriptsize, fill=white, inner sep=0.8pt, text=red}]
 \begin{scope}
  \node[elt] (b1234) at (0,0)        {$1234$};
  \node[elt] (b1243) at (-2,1.67)    {$1243$};
  \node[elt] (b1324) at (0,1.67)     {$1324$};
  \node[elt] (b2134) at (2,1.67)     {$2134$};
  \node[elt] (b1423) at (-2,3.33)    {$1423$};
  \node[elt] (b2143) at (0,3.33)     {$2143$};
  \node[elt] (b2314) at (2,3.33)     {$2314$};
  \node[elt] (b2413) at (0,5)        {$2413$};
  \draw        (b1234) -- node[lab,pos=.5]     {$t_{34}$} (b1243);
  \draw[flip]  (b1234) -- node[fliplab,pos=.5] {$t_{23}$} (b1324);
  \draw        (b1234) -- node[lab,pos=.5]     {$t_{12}$} (b2134);
  \draw[flip]  (b1243) -- node[fliplab,pos=.5] {$t_{23}$} (b1423);
  \draw        (b1243) -- node[lab,pos=.3]     {$t_{12}$} (b2143);
  \draw[flip]  (b1324) -- node[fliplab,pos=.3] {$t_{24}$} (b1423);
  \draw[flip]  (b1324) -- node[fliplab,pos=.3] {$t_{13}$} (b2314);
  \draw        (b2134) -- node[lab,pos=.3]     {$t_{34}$} (b2143);
  \draw[flip]  (b2134) -- node[fliplab,pos=.5] {$t_{23}$} (b2314);
  \draw[flip]  (b1423) -- node[fliplab,pos=.5] {$t_{13}$} (b2413);
  \draw[flip]  (b2143) -- node[fliplab,pos=.5] {$t_{23}$} (b2413);
  \draw[flip]  (b2314) -- node[fliplab,pos=.5] {$t_{24}$} (b2413);
  \node at (0,-0.75) {$[e,\,2413]$};
 \end{scope}
 \begin{scope}[shift={(5.5,0)}]
  \node[elt] (m3142) at (0,0)        {$3142$};
  \node[elt] (m4132) at (-2,1.67)    {$4132$};
  \node[elt] (m2143) at (0,1.67)     {$2143$};
  \node[elt] (m3241) at (2,1.67)     {$3241$};
  \node[elt] (m2134) at (-2,3.33)    {$2134$};
  \node[elt] (m4231) at (0,3.33)     {$4231$};
  \node[elt] (m1243) at (2,3.33)     {$1243$};
  \node[elt] (m1234) at (0,5)        {$1234$};
  \draw[->]        (m3142) -- (m4132);
  \draw[->]        (m3142) -- (m3241);
  \draw[->]        (m4132) -- (m4231);
  \draw[->]        (m3241) -- (m4231);
  \draw[flip,->]   (m2143) -- (m3142);
  \draw[flip,->]   (m2134) -- (m4132);
  \draw[flip,->]   (m2134) -- (m2143);
  \draw[flip,->]   (m1243) -- (m2143);
  \draw[flip,->]   (m1243) -- (m3241);
  \draw[flip,->]   (m1234) -- (m2134);
  \draw[flip,->]   (m1234) -- (m4231);
  \draw[flip,->]   (m1234) -- (m1243);
  \node at (0,-0.75) {$[e,\,2413]\,x^{-1}$};
 \end{scope}
 \begin{scope}[shift={(10,0)}]
  \node[elt] (a1234) at (0,0)     {$1234$};
  \node[elt] (a2134) at (-1,1)    {$2134$};
  \node[elt] (a1243) at (1,1)     {$1243$};
  \node[elt] (a2143) at (0,2)     {$2143$};
  \node[elt] (a3142) at (0,3)     {$3142$};
  \node[elt] (a4132) at (-1,4)    {$4132$};
  \node[elt] (a3241) at (1,4)     {$3241$};
  \node[elt] (a4231) at (0,5)     {$4231$};
  \draw (a1234) -- node[lab,pos=.5] {$t_{12}$} (a2134);
  \draw (a1234) -- node[lab,pos=.5] {$t_{34}$} (a1243);
  \draw (a2134) -- node[lab,pos=.5] {$t_{34}$} (a2143);
  \draw (a1243) -- node[lab,pos=.5] {$t_{12}$} (a2143);
  \draw (a2143) -- node[lab,pos=.5] {$t_{14}$} (a3142);
  \draw (a3142) -- node[lab,pos=.5] {$t_{13}$} (a4132);
  \draw (a3142) -- node[lab,pos=.5] {$t_{24}$} (a3241);
  \draw (a4132) -- node[lab,pos=.5] {$t_{24}$} (a4231);
  \draw (a3241) -- node[lab,pos=.5] {$t_{13}$} (a4231);
  \node at (0,-0.75) {$[e,2413]\,(2413)^{-1}$};
 \end{scope}
\end{tikzpicture}
    \captionsetup{width=1.0\linewidth}
  \captionof{figure}{The interval $[e,2413]$ with flipping edges in red
  (left), the shift of its elements by $x = 2413$ with edges directed
  toward Bruhat-larger elements (middle), and the Hasse diagram of
  $[e,2413]\,(2413)^{-1}$ (right).  Note that the rank sequence $1,2,1,1,2,1$ is not
  unimodal. See Example~\ref{ex:bowtie}.}
  \label{fig:bowtie}
 \end{center}
\end{figure}

We highlight two facts about the poset $[e,2341]\,(3214)^{-1}$ of Example~\ref{ex:longstem} that are not visible from the construction.   First, although
the shift removes covers and creates new ones, every maximal chain of $[e,2341]\,(3214)^{-1}$ still has length $5$.  We will prove in Section~\ref{sub:gradedness} that shifted lower intervals are graded.  Second, under the $N(3214)$-initial reflection ordering given in Equation \eqref{eqn:S4_reflection_order_example}, every interval of this poset has a unique increasing maximal chain with respect to the right labels shown in Figure \ref{fig:longstemex}.  For example, for the whole poset the maximal chain is
$$1324 \lessdot 3124 \lessdot 3214 \lessdot 4213 \lessdot 4312 \lessdot
4321,$$ with label sequence $t_{12}, t_{23}, t_{14}, t_{24}, t_{34}$.
We prove in Section \ref{section:EL_labels} that the right labels give an EL-labeling of shifted (by $x$) lower-intervals equipped with $N(x)$-initial reflection orderings.

We conclude this section with an important identity we call the \newword{switching lemma}.  The switching lemma governs which edges in a Bruhat interval get ``reversed" by shifting (note the red edges in Figure~\ref{fig:longstemex}).

\begin{lemma}[Switching lemma]\label{lem:switch}
Let $x \in W$. For every $u \in W$,
\[
  \{\, t \in T : ux^{-1}t < ux^{-1} \,\}
  \;=\; N(x) \,\Delta\, x\,\{\, r \in T : ur < u \,\}\,x^{-1}.
\]
That is, for $t \in T$, writing $r = x^{-1}tx$ (so that
$ux^{-1}t = (ur)x^{-1}$),
\[ur <u \iff
  \begin{cases}
    ux^{-1}t < ux^{-1} & \text{if } t \notin N(x),\\ 
    ux^{-1}t > ux^{-1} & \text{if } t \in N(x).
  \end{cases}
\] 
In other words, shifting by $x$ reverses the Bruhat comparison of the pair
$\{u,\, ur\}$ if and only if the label $t = xrx^{-1}$ of the shifted pair lies in $N(x)$.
\end{lemma}

\begin{proof}
For $u \in W$ and $t \in T$ we have $ut < u$ if and only if  $t \in T_R(u)$. Applying
the cocycle identity \eqref{eqn:cocycleR} to the product $ux^{-1}$ gives
$T_R(ux^{-1}) = T_R(x^{-1}) \,\Delta\, x\,T_R(u)\,x^{-1}$, and
$T_R(x^{-1}) = N(x)$ by definition.
\end{proof}

For example, in Figure~\ref{fig:longstemex} we have $x = 3214$ and $N(x) = \{t_{12},t_{13},t_{23}\}$. The red edges (the edges that switched direction going from the left panel to the middle one) were exactly those in $\{t_{12}, t_{13}, t_{23}\}$, as guaranteed by the switching lemma.

\section{Properties of shifted lower intervals}\label{section:shift_lower_intervals}
In this section we show that, for an arbitrary Coxeter group, the shifted lower interval $\swx$ has unique extrema and that all of its maximal chains have the same length. For the extrema we use the \newword{Demazure product} \cite{Dem74} together with a ``downward" counterpart that we call the \newword{opposite Demazure operator}. For gradedness we use the \newword{weak generalized lifting property} of Caselli, D'Adderio, and Marietti \cite{Caselli-Adderio-Marietti21}.

\subsection{The unique maximum and minimum of shifted Bruhat intervals}
There is an associative operation on Coxeter groups called the \newword{Demazure product} \cite{Dem74} (also $0$-\newword{Hecke} or \newword{greedy product}).

The \newword{Coxeter monoid} structure (also called the $0$-\newword{Iwahori-Hecke monoid} structure) on $W$ is defined to be the monoid generated by $S$ with a product $\star$ satisfying the Coxeter braid relations for $s\neq t$ along with the relation $s\star s=s$ for all $s\in S$ (this new relation replaces $s^2=id$ in the usual product).  This monoid product is what we refer to as the Demazure product and was first studied by Norton in \cite{No79} in the context of Hecke algebras.  It is well known that, as sets, $W=\langle S,\star\rangle$.  We say an expression $w=s_1\cdots s_k$ is reduced if $\ell(w)=k$.  Recall that if $\ell(w)=k$, then $w$ cannot be expressed with fewer than $k$ generators in $S$. The next lemma records some basic facts about the Coxeter monoid.

\begin{lemma}\cite[Lemma 1.3, Corollary 1.4]{No79}
\label{lem:monoid_properties}
Let $W$ be a Coxeter group with generating set $S$.  Then the following are true:
\begin{enumerate}
    \item Let $(s_1,\ldots,s_k)$ be a sequence of generators in $S$.  Then
    $$s_1\cdots s_k\leq s_1\star\cdots \star s_k$$ with equality if and only if $(s_1,\ldots,s_k)$ is a reduced expression.
    \item For any $s\in S$ and $w\in W$,
$$s\star w=\begin{cases} w & \text{if}\quad \ell(sw)< \ell(w)\\ sw & \text{if}\quad \ell(sw)>\ell(w).\end{cases}$$
\end{enumerate}
\end{lemma}

As a consequence, there is a very nice interpretation of $w \star u$ for any $w,u\in W$ in terms of Bruhat intervals. 

\begin{proposition}\cite[Lemma 1]{He09}, \cite[Proposition 8]{Kenny14}\label{prop:interval_product}
For any $w,u\in W$, the lower interval $$[e,w\star u]=\{yz \ |\ y\in [e,w], z\in[e,u]\}.$$
\end{proposition}

From this we can describe the maximal element in a shifted lower interval.

\begin{proposition}
\label{prop:max}
For any Coxeter group $W$ and $w, x \in W$, the shifted lower interval $\swx$ has a unique maximum:
\[
  \max(\swx) \;=\; w \star x^{-1}.
\]
\end{proposition}
\begin{proof}
By \cite[Lemma~1]{He09}, we can write $w \star x^{-1} = ux^{-1}$ for some $u \leq w$ and hence $w \star x^{-1} \in \swx$. For any $v \leq w$, Proposition~\ref{prop:interval_product} gives
$vx^{-1} \in [e, w \star x^{-1}]$.  In other words, $vx^{-1} \leq w \star x^{-1}$.
\end{proof}

For the minimum we give a closed formula valid in every Coxeter group, based on the following downward counterpart of the Demazure product.

\begin{definition}
For $u \in W$ and $s\in S$, we first define the operator:

$$u\bstar s:=\begin{cases} u & \text{if}\quad \ell(us)> \ell(u)\\ us & \text{if}\quad \ell(us)<\ell(u)\end{cases}.$$

Let $\sigma = s_1 \cdots s_k$ be a word in the alphabet $S$, not necessarily reduced.  We define the \newword{opposite Demazure operator} as the expression 
$$u \bstar \sigma:=(((u\bstar s_1)\bstar s_2)\cdots \bstar s_k).$$ 
For $w \in W$, we write $u \bstar w$ for $u \bstar \sigma$ with $\sigma$ a reduced word of $w$ (we will later see that  $u \bstar w$ is independent of the choice of reduced word of $w$).
\end{definition}

Note that $u \bstar s \le u$ and $u \bstar s \le us$ hold by construction. Allowing non-reduced words costs nothing and is needed for the deletion argument of Lemma~\ref{lem:bstaranti}, since deleting a letter from a reduced word can leave a non-reduced one. For a non-reduced word the opposite Demazure operator depends on the word itself rather than on its product, which is why the element notation $u \bstar w$ is reserved for reduced words. We show in Corollary~\ref{cor:bstarwelldef} that the result of the opposite Demazure operator does not depend on the chosen reduced word.

\begin{remark}
    Unlike the usual Demazure product, the opposite Demazure operator does not give rise to a monoid structure on $W$.  In particular, it is not associative, even when restricting to reduced words.  The operator is constructed using the right-action of simple reflections, so one could define a left-action analogue.  We remark that, also unlike the usual Demazure product, the right and left operator do not agree in the sense that $u\bstar w$ takes different values under these actions.  For example, $u\bstar w$ is always less than or equal to $u$ under the right action and less than $w$ under the left action.  This incompatibility implies that $\bstar$ does not behave well with respect to inverses, which we will see later. 
\end{remark}

Despite the lack of monoid structure, the opposite Demazure operator does satisfy some nice properties which we now show.

\begin{lemma}\label{lem:bstarmono}
If $u \geq w$, then $u \bstar s \geq w \bstar s$. Consequently $u \geq w$ implies $u \bstar \sigma \geq w \bstar \sigma$ for any word $\sigma$.
\end{lemma}
\begin{proof}
If $us > u$, then $u \bstar s = u \geq w \geq w \bstar s$. If $us < u$ and $ws < w$, then the lifting property gives $us \geq ws$. If $us < u$ and $ws > w$, then the lifting property gives $w \leq us$, and in this case $w \bstar s = w$. The lemma follows by iterating on the word $\sigma$.
\end{proof}

Next we observe how the second entry in the opposite Demazure operator receives a twist in the ordering compared to the usual Demazure product:

\begin{lemma}\label{lem:bstaranti}
Let $\sigma$ be a word in the alphabet $S$ and let $\tau$ be a subword of $\sigma$. Then $u \bstar \sigma \leq u \bstar \tau$.
\end{lemma}
\begin{proof}
It suffices to consider the case where $\tau$ is obtained from $\sigma$ by deleting a single letter $s$, and then iterate. The two computations agree up to the position of $s$, arriving at some common element $m$. At that position the computation along $\sigma$ moves to $m \bstar s \leq m$ while the computation along $\tau$ stays at
$m$.  The lemma now follows from Lemma~\ref{lem:bstarmono}.
\end{proof}

\begin{proposition}
\label{prop:minformula}
Let $W$ be an arbitrary Coxeter group and $x, w \in W$. Then $x \bstar w$, computed with any reduced word of $w$, is the minimum of the (left) shifted interval 
\[
  x\,[e,w] \;=\; \{\, xv : v \le w \,\}.
\]
\end{proposition}
\begin{proof}
Fix the reduced word $\sigma = s_1 \cdots s_k$ of $w$ used to compute $x \bstar w$. The letters applied while running through $\sigma$ multiply, in order, to a subword of $\sigma$, so $x \bstar w = xw'$ for some $w' \leq w$, and hence $x \bstar w \in x\,[e,w]$.

Let $v \leq w$.  We claim that $x \bstar \tau \leq xv$ for any reduced word $\tau = s_{i_1} \cdots s_{i_j}$ of $v$ by induction on $\ell(v)$.  First observe that the case $v = e$ is trivial.  For $v\neq e$, we write $v = v's$ with $s = s_{i_j}$ and $\ell(v') < \ell(v)$.  The induction hypothesis and Lemma~\ref{lem:bstarmono} gives $$x \bstar \tau = (x \bstar (s_{i_1} \cdots s_{i_{j-1}})) \bstar s \leq (xv') \bstar s \leq xv's = xv.$$
By the subword property we may choose $\tau$ to be a subword of $\sigma$.  Lemma~\ref{lem:bstaranti} yields
$x \bstar w \leq x \bstar \tau \leq xv$. So $x \bstar w$ is below every element of $x\,[e,w]$.
\end{proof}

Proposition~\ref{prop:minformula} also explains why the two arguments of $\bstar$ cannot be exchanged the way they can for $\star$: the operator computes the minimum of $x\,[e,w]$, a minimum over the second factor only, so in general $(x \bstar w)^{-1} \neq w^{-1} \bstar x^{-1}$ (already for $x = e$ and $w = s$, where the two sides are $e$ and $s$).

\begin{corollary}\label{cor:bstarwelldef}
The opposite Demazure operator $x \bstar w$ does not depend on the choice of reduced word of $w$.
\end{corollary}

\begin{proof}
For every choice of reduced word, Proposition~\ref{prop:minformula} identifies $x \bstar w$ with the minimum of $x\,[e,w]$, which is defined without reference to any word.
\end{proof}

\begin{corollary}
\label{prop:uniqueminmax}
For any Coxeter group $W$ and $w, x \in W$, the shifted lower interval $\swx$ has a unique minimum and a unique maximum:
\[
  \min \swx \;=\; \bigl(x \bstar w^{-1}\bigr)^{-1}
  \qquad \text{and} \qquad
  \max \swx \;=\; w \star x^{-1}.
\]
\end{corollary}
\begin{proof}
The maximum is Proposition~\ref{prop:max}. For the minimum, the map $w \mapsto w^{-1}$ is an automorphism of the Bruhat order carrying the translate $x\,[e,w^{-1}]$ onto $\swx$, so it carries the minimum $x \bstar w^{-1}$ of Proposition~\ref{prop:minformula} to the minimum of $\swx$.
\end{proof}

When $W$ is finite, the minimum can instead be obtained from the maximum: right multiplication by $w_0$ carries $\swx$ onto the shifted lower interval $[e,w]\,(w_0x)^{-1}$ and, being an anti-automorphism of the Bruhat order, reverses the induced order. Hence
$\min \swx = \bigl( w \star (x^{-1}w_0) \bigr)\, w_0$, and comparing with Corollary~\ref{prop:uniqueminmax} yields the identity
$$\bigl(x \bstar w^{-1}\bigr)^{-1} = \bigl( w \star (x^{-1}w_0) \bigr)\, w_0.$$

\begin{remark}
For the type A case, one could show that $\sywx$ has unique maximum and minimum using Bruhat interval polytopes by Tsukerman-Williams \cite{TW15}. We can construct a linear functional that extends the partial order we get from transpositions increasing length, and permuting the coordinates according to $x$ would give us the shifted order. Then the polytope being convex would imply there is a unique minimum and maximum.
\end{remark}

\subsection{Gradedness of shifted intervals}
\label{sub:gradedness}
In the previous subsection, we prove that shifted lower intervals have unique maximum and minimum.  To prove that these posets are graded, it suffices to show that all maximal chains have the same length.  The main tool we use is the \newword{weak generalized lifting property} of Caselli, D'Adderio, and Marietti \cite{Caselli-Adderio-Marietti21}.  We remark that our original approach used the \newword{generalized lifting property} of Tsukerman-Williams \cite{TW15} and Caselli-Sentinelli \cite{CS17}, which holds for Bruhat intervals in finite and simply-laced Coxeter groups \cite[Theorem~5.9]{CS17}.  The main advantage of the weak generalized lifting property is that it holds for every Coxeter system.

\newcommand{\acone}{\alpha\mathrm{Cone}}
Recall from Section~2 that every reflection $t \in T$ has a corresponding positive root $\alpha_t \in \Phi^+$. Given some subset $A$ of reflections, let 
\[\acone(A) \;:=\; \mathrm{Cone}\{\alpha_t \mid t \in A\},\]
the cone (set of all non-negative sums) of the corresponding roots.

\begin{theorem}[Weak generalized lifting property \cite{Caselli-Adderio-Marietti21}]\label{thm:weak_lifting}
    Let $(W,S)$ be an arbitrary Coxeter system with reflection set $T$, and let $u < v$ in $W$. Define 
    \[
    R_u = \{t \in T \mid u \lessdot ut \leq v\} \hspace{1cm}\text{and}\hspace{1cm} R^v = \{t\in T \mid u \leq vt \lessdot v\}.
    \]
    Then $\acone(R_u) \cap \acone(R^v) \neq \{0\}$.
\end{theorem}

In \cite{Caselli-Adderio-Marietti21} the theorem is stated with left multiplication, for the sets $\{t : u \lessdot tu \leq v\}$ and $\{t : u \leq tv \lessdot v\}$. Applying that statement to the interval $[u^{-1}, v^{-1}]$ and using that inversion is an automorphism of the Bruhat order with $w \lessdot tw$ if and only if $w^{-1} \lessdot w^{-1}t$ yields the form above, with the same reflections and hence the same roots.

\begin{proposition}
\label{prop:unitchains}
Let $W$ be a Coxeter group and $w, x \in W$. Then for all $u < v$ in $\swx$ there is a chain from $u$ to $v$ in $\swx$ of
length $\ell(v) - \ell(u)$.
\end{proposition}
\begin{proof}
    We proceed by induction on $N := \ell(v) - \ell(u)$. If $N = 1$, then $u<v$ is itself such a chain. Let $N \geq 2$ and assume the claim for all smaller values. 
We first argue that $ut \in \swx$ for some $t \in R_u$, or $vr \in \swx$ for some $r \in R^v$. For $t \in R_u$, it suffices to have $utx < ux$, since then $utx < ux \leq w$. Likewise, for $r \in R^v$, it suffices to have $vrx < vx$. Assume for the sake of contradiction that
    \[ux < utx \hspace{1cm}\text{and}\hspace{1cm}vx < vrx\]
    for all $t \in R_u$ and $r \in R^v$.
    
    Fix $t \in R_u$. The pair $\{u, ut\}$ has $u < ut$, and by Lemma~\ref{lem:switch} the shift preserves this comparison, meaning $ux < utx$, exactly when $t \in T \setminus N(x)$. Similarly, for $r \in R^v$ the pair $\{vr, v\}$ has $vr < v$, and the shift reverses this comparison, meaning $vx < vrx$, exactly when $r \in N(x)$. Our assumption is thus that 
    \[R_u \subseteq T \setminus N(x) \hspace{1cm}\text{and}\hspace{1cm}R^v \subseteq N(x),\]
    and therefore 
    \[
        \acone(R_u) \cap \acone(R^v) \subseteq \acone(T\setminus N(x)) \cap \acone(N(x)).
    \]
    We claim the right-hand side is $\{0\}$. If $\sum c_i \alpha_i \in \acone(N(x))$, then 
    \[x^{-1} \left(\sum c_i \alpha_i\right) = \sum c_i x^{-1}(\alpha_i) \in \mathrm{Cone}(\Phi^-),\]
    since $x^{-1}(\alpha_t) \in \Phi^-$ for every $t \in N(x)$.  If $\sum d_i \beta_i \in \acone(T \setminus N(x))$, then 
    \[x^{-1} \left(\sum d_i \beta_i\right) = \sum d_i x^{-1}(\beta_i) \in \mathrm{Cone}(\Phi^+).\]
    Recall that $\mathrm{Cone}(\Phi^+)$ and $\mathrm{Cone}(\Phi^-)$ intersect trivially. Since $x^{-1}$ acts as a linear automorphism of the ambient space, we get that 
    \[\acone(N(x)) \cap \acone(T\setminus N(x)) = \{0\}.\]
    This contradicts Theorem~\ref{thm:weak_lifting}, and so there must be some $t \in R_u$ with $ut \in \swx$ or some $r \in R^v$ with $vr \in \swx$.

    In the first case, $\ell(v) - \ell(ut) = N - 1$ and $ut \le v$, so by induction there is a chain from $ut$ to $v$ in $\swx$ of length $N-1$, and prepending $u < ut$ yields the desired chain. The second case is symmetric, appending $vr < v$ to a chain from $u$ to $vr$ of length $N-1$.
\end{proof}
    
The above Proposition tells us that $\swx$ is a graded poset. Note where the proof of Proposition~\ref{prop:unitchains} uses that the
interval is a lower one: membership of $ut$ in $\swx$ only
requires the upper bound $utx \le w$, which follows from $utx < ux$ by transitivity. For a middle interval, the element $utx$ must also stay above the lower bound.  Indeed, if the shift by $x^{-1}$ pushes $utx$ down, $ut$ may fall below the lower bound. This failure is illustrated in the following example.

\begin{example}\label{ex:nongraded}
Let $W = S_4$ and $x = 1243 = x^{-1}$. The middle shifted interval $[1243, 3412]\,x^{-1}$ has $8$ elements, a unique minimum $1234$ and a unique maximum $3421$ (see Figure~\ref{fig:nongraded}).  However it is not graded: it contains the maximal chains
\[
  1234 < 1324 < 3124 < 3421
\]
of length $3$ and
\[
  1234 < 1324 < 1423 < 1432 < 2431 < 3421
\]
of length $5$. In particular, unlike in the lower case, covers of a middle shifted interval need not be covers of $W$: here $3124 \lessdot 3421$ has length gap $3$.
\end{example}

In Figure~\ref{fig:nongraded}, the left panel shows the interval $[1243, 3412]$ in $S_4$, with each cover marked by its right label.  The unique red edge is the one whose Bruhat comparison is reversed by the shift $x = 1243$, namely the edge with label $t_{34}$, since $N(x) = \{t_{34}\}$ and $xt_{34}x^{-1} = t_{34}$.  

The middle panel shows each vertex $u$ replaced by $ux^{-1}$, edges directed toward the Bruhat-larger element. 

The right panel shows the Hasse diagram of the middle shifted interval of Example~\ref{ex:nongraded}, drawn with each vertex at height equal to its Coxeter length. It has a unique minimum and maximum but is not graded: its maximal chains have lengths $3$ and $5$. Dashed edges are the covers $2134 \lessdot 2431$ and $3124 \lessdot 3421$ of length gap $3$, which cannot occur in a shifted lower interval by Proposition~\ref{prop:unitchains}.

\begin{figure}[h]
 \begin{center}
  \vspace{3.7mm}
\begin{tikzpicture}[scale=0.82,
  every node/.style={font=\small},
  elt/.style={inner sep=1.5pt},
  lab/.style={font=\scriptsize, fill=white, inner sep=0.8pt},
  flip/.style={red, thick},
  fliplab/.style={font=\scriptsize, fill=white, inner sep=0.8pt, text=red},
  longcov/.style={dashed}]
 \begin{scope}
  \node[elt] (b1243) at (0,0)        {$1243$};
  \node[elt] (b1342) at (-2,1.67)    {$1342$};
  \node[elt] (b1423) at (0,1.67)     {$1423$};
  \node[elt] (b2143) at (2,1.67)     {$2143$};
  \node[elt] (b1432) at (-2,3.33)    {$1432$};
  \node[elt] (b3142) at (0,3.33)     {$3142$};
  \node[elt] (b2413) at (2,3.33)     {$2413$};
  \node[elt] (b3412) at (0,5)        {$3412$};
  \draw        (b1243) -- node[lab,pos=.5]     {$t_{24}$} (b1342);
  \draw        (b1243) -- node[lab,pos=.5]     {$t_{23}$} (b1423);
  \draw        (b1243) -- node[lab,pos=.5]     {$t_{12}$} (b2143);
  \draw[flip]  (b1423) -- node[fliplab,pos=.3] {$t_{34}$} (b1432);
  \draw        (b1342) -- node[lab,pos=.5]     {$t_{23}$} (b1432);
  \draw        (b1342) -- node[lab,pos=.3]     {$t_{12}$} (b3142);
  \draw        (b1423) -- node[lab,pos=.3]     {$t_{13}$} (b2413);
  \draw        (b2143) -- node[lab,pos=.3]     {$t_{14}$} (b3142);
  \draw        (b2143) -- node[lab,pos=.5]     {$t_{23}$} (b2413);
  \draw        (b1432) -- node[lab,pos=.5]     {$t_{13}$} (b3412);
  \draw        (b3142) -- node[lab,pos=.5]     {$t_{23}$} (b3412);
  \draw        (b2413) -- node[lab,pos=.5]     {$t_{14}$} (b3412);
  \node at (0,-0.75) {$[1243,\,3412]$};
 \end{scope}
 \begin{scope}[shift={(5.5,0)}]
  \node[elt] (m1234) at (0,0)        {$1234$};
  \node[elt] (m1324) at (-2,1.67)    {$1324$};
  \node[elt] (m1432) at (0,1.67)     {$1432$};
  \node[elt] (m2134) at (2,1.67)     {$2134$};
  \node[elt] (m1423) at (-2,3.33)    {$1423$};
  \node[elt] (m3124) at (0,3.33)     {$3124$};
  \node[elt] (m2431) at (2,3.33)     {$2431$};
  \node[elt] (m3421) at (0,5)        {$3421$};
  \draw[->]        (m1234) -- (m1324);
  \draw[->]        (m1234) -- (m1432);
  \draw[->]        (m1234) -- (m2134);
  \draw[->]        (m1324) -- (m1423);
  \draw[->]        (m1324) -- (m3124);
  \draw[->]        (m1432) -- (m2431);
  \draw[->]        (m2134) -- (m3124);
  \draw[->]        (m2134) -- (m2431);
  \draw[->]        (m1423) -- (m3421);
  \draw[->]        (m3124) -- (m3421);
  \draw[->]        (m2431) -- (m3421);
  \draw[flip,->]   (m1423) -- (m1432);
  \node at (0,-0.75) {$[1243,\,3412]\,x^{-1}$};
 \end{scope}
 \begin{scope}[shift={(10.4,0)}]
  \node[elt] (n1234) at (0,0)      {$1234$};
  \node[elt] (n1324) at (-1,1)     {$1324$};
  \node[elt] (n2134) at (1,1)      {$2134$};
  \node[elt] (n1423) at (-1.6,2)   {$1423$};
  \node[elt] (n3124) at (0.3,2)    {$3124$};
  \node[elt] (n1432) at (-1.6,3)   {$1432$};
  \node[elt] (n2431) at (0.9,4)    {$2431$};
  \node[elt] (n3421) at (0.3,5)    {$3421$};
  \draw (n1234) -- node[lab,pos=.5]  {$t_{23}$} (n1324);
  \draw (n1234) -- node[lab,pos=.5]  {$t_{12}$} (n2134);
  \draw (n1324) -- node[lab,pos=.5]  {$t_{24}$} (n1423);
  \draw (n1324) -- node[lab,pos=.5]  {$t_{12}$} (n3124);
  \draw (n2134) -- node[lab,pos=.5]  {$t_{13}$} (n3124);
  \draw (n1423) -- node[lab,pos=.5]  {$t_{34}$} (n1432);
  \draw (n1432) -- node[lab,pos=.5]  {$t_{14}$} (n2431);
  \draw (n2431) -- node[lab,pos=.5]  {$t_{13}$} (n3421);
  \draw[longcov] (n2134) -- node[lab,pos=.5]  {$t_{24}$} (n2431);
  \draw[longcov] (n3124) -- node[lab,pos=.35] {$t_{24}$} (n3421);
  \node at (0,-0.75) {$[1243,3412]\,(1243)^{-1}$};
 \end{scope}
\end{tikzpicture}
    \captionsetup{width=1.0\linewidth}
  \captionof{figure}{The interval $[1243,3412]$ with flipping edges in red
  (left), the shift of its elements by $x = 1243$ with edges directed
  toward the Bruhat-larger element (middle), and the Hasse diagram of
  $[1243,3412]\,(1243)^{-1}$ (right). See Example~\ref{ex:nongraded}.}
  \label{fig:nongraded}
 \end{center}
\end{figure}
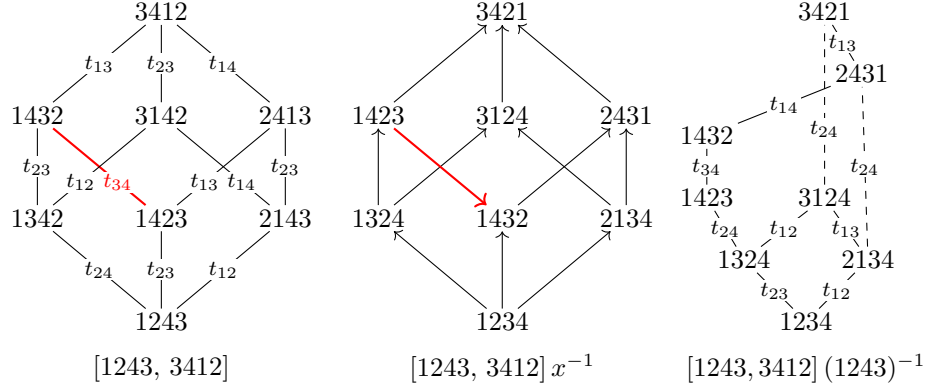

\section{EL-labelings of shifted intervals}\label{section:EL_labels}

The goal of this section is to prove Theorem \ref{thm:main_intro}.  In particular, we will show that for every Coxeter group, the right reflection labeling $\lambda_R$ is an EL-labeling of the shifted lower interval $\swx$ under any $N(x)$-initial reflection ordering (see Theorem \ref{thm:main} below). 

Our proof splits into two independent halves. The first half is a statement on general shifted intervals: provided the cover relations of a shifted interval come from the ambient Bruhat order, the entire EL-shellability framework reduces to a local check on its rank-$2$ intervals. The second half uses the switching lemma (Lemma \ref{lem:switch}) to verify that shifted lower intervals always pass this local check when equipped with an $N(x)$-initial reflection ordering.

\subsection{Ambient covers and broken diamonds}

Because a shifted interval $\suvx$ is defined as an induced subposet of $(W, \le)$, it inherits its partial order directly from the standard Bruhat order, which we call the \newword{ambient} Bruhat order.

However, it does not automatically inherit its cover relations. 

A \newword{cover relation} $y \lessdot z$ in a poset $P$ simply means that $y < z$ and no other element of $P$ lies strictly between them. The cover relations in Bruhat order $(W, \le)$ are always between elements whose lengths differ by one, but in an arbitrary induced subposet, or indeed some $\suvx$, a cover relation may skip over several intermediate lengths (as seen in the dashed edges of Figure~\ref{fig:nongraded}). When working with an interval $[y,z]$ in $P\subseteq W$, we use the notation $[y,z]_W$ to denote the corresponding interval in the poset $W$.

We say that the covers of $\suvx$ are \newword{ambient} if every cover relation $y \lessdot z$ in $\suvx$ is also a cover relation in the ambient Bruhat order, that is, $\ell(z) = \ell(y) + 1$. In this case the right label $\lambda_R(y \lessdot z) = y^{-1}z$ of every cover is a reflection.  

Proposition~\ref{prop:unitchains} from the previous section shows that shifted lower intervals are always well-behaved in this sense:

\begin{lemma}\label{lem:ambientcovers}
Let $W$ be a Coxeter group and $w, x \in W$. Then the covers of $\swx$ are ambient in $W$.
\end{lemma}
\begin{proof}
Suppose $u \lessdot v$ in $\swx$ with $\ell(v) - \ell(u) \ge 2$. Then the chain provided by Proposition~\ref{prop:unitchains} passes strictly
between $u$ and $v$ inside $\swx$, contradicting the cover relation.
\end{proof}
 
When the covers of $\suvx$ are ambient, its rank-$2$ intervals become tightly constrained. Recall that the Bruhat order is \newword{thin}: if two elements $y < z$ satisfy $\ell(z) - \ell(y) = 2$, there are exactly two elements of $W$ strictly between them \cite[Lemma~2.7.3]{Bjorner-Brenti05}. In our induced subposet, thinness yields the following structural property of rank-$2$ intervals:

\begin{proposition}\label{prop:diamonds}
Let $\suvx$ be a shifted interval whose covers are ambient in $W$. Let $[y,z]$ be a rank-$2$ interval in $\suvx$, and let $m_1, m_2 \in W$ be the two ambient middle elements satisfying $y \lessdot m_i \lessdot z$ in $W$. Then exactly one of the following holds:
\begin{enumerate}
    \item \textbf{Full Diamond:} both $m_1$ and $m_2$ belong to $\suvx$.
    \item \textbf{Broken Diamond:} exactly one of the middle elements, say $m_1$, belongs to $\suvx$. In this case the unique maximal chain $y \lessdot m_1 \lessdot z$ of $[y,z]$ in $\suvx$ is called its \newword{surviving chain}, and the missing element $m_2$ is called its \newword{dead middle}.
\end{enumerate}
\end{proposition}
\begin{proof}
Since $[y,z]$ has rank $2$ in $\suvx$, it contains at least one maximal chain $y \lessdot b \lessdot z$. Because the covers are ambient, this chain satisfies $\ell(b) = \ell(y)+1$ and $\ell(z) = \ell(b)+1$, so $b$ is one of the two ambient middle elements. The classification then depends on whether the remaining ambient middle element belongs to $\suvx$.
\end{proof}

\subsection{Assigning EL-labelings}
With the diamond dichotomy in place, we can build the general tool that powers our shellability result. It shows that when the covers of a shifted interval are ambient, we do not need to inspect the global poset to prove EL-shellability. Instead, we only need to inspect how labels behave on the surviving chains of its broken diamonds.

\begin{proposition}    
\label{prop:mainmid}
Let $\suvx$ be a shifted interval whose covers are ambient in $W$ and suppose $\suvx$ has a unique minimum and a unique maximum. Let $\prec$ be a reflection ordering on $T$. If the surviving chain of every broken diamond in $\suvx$ is $\prec$-increasing, then the right reflection labeling $\lambda_R$ is an EL-labeling of $\suvx$.
\end{proposition}

\begin{proof}
Fix an interval $[y,z]$ of $\suvx$. We must show that it contains a unique $\prec$-increasing maximal chain, and that this chain lexicographically precedes all other maximal chains of $[y,z]$.

Because the covers of $\suvx$ are ambient, any maximal chain
$y = c_0 \lessdot c_1 \lessdot \cdots \lessdot c_k = z$ of $[y,z]$ in $\suvx$ satisfies $\ell(c_i) = \ell(y) + i$. Hence, it is a maximal chain of the ambient Bruhat interval $[y,z]_W$. 

By Dyer's theorem (Theorem~\ref{thm:dyer}, applied
to right labels), the interval $[y,z]_W$ possesses exactly one
$\prec$-increasing maximal chain. Consequently, $[y,z]$ contains at most one $\prec$-increasing maximal chain. 

To prove that an increasing chain exists, let
$C: y = c_0 \lessdot c_1 \lessdot \cdots \lessdot c_k = z$ be a maximal chain of $[y,z]$ in $\suvx$ that is lexicographically first among all its maximal chains. Assume for the sake of contradiction that $C$ contains a descent, meaning $\lambda_R(c_{i-1} \lessdot c_i) \succ \lambda_R(c_i \lessdot c_{i+1})$ for some index $i$. The interval $[c_{i-1}, c_{i+1}]$ is a rank-$2$ interval of $\suvx$ containing the chain $c_{i-1} \lessdot c_i \lessdot c_{i+1}$, so Proposition~\ref{prop:diamonds} applies.

If $[c_{i-1}, c_{i+1}]$ is a broken diamond, then the chain through $c_i$, the surviving chain, is $\prec$-increasing by the hypothesis, which is a contradiction.

Otherwise, the only other possibility allowed by Proposition~\ref{prop:diamonds} is that $[c_{i-1}, c_{i+1}]$ is a full diamond. Let $m$ be its other middle element, which belongs to $\suvx$. In the ambient Bruhat order, the diamond $[c_{i-1}, c_{i+1}]_W$ has exactly two maximal chains, of
which exactly one is increasing.  Theorem~\ref{thm:dyer} guarantees that the increasing chain is lexicographically smaller than the decreasing one. We can replace $c_i$ with $m$ to get a lexicographically smaller chain, again giving us a contradiction.

Therefore the lexicographically first chain $C$ contains no descent, hence it is $\prec$-increasing. Combined with the first part of this proof, $C$ is the unique increasing maximal chain of $[y,z]$ and is lexicographically first, completing the proof.
\end{proof}

Note that the proof of Proposition~\ref{prop:mainmid} never uses that $\suvx$ is a shifted interval. The same proof holds for any subset $P \subseteq W$, with the induced order, all of whose covers are ambient in $W$. We state it for shifted intervals as these are the only posets we apply it to.

\subsection{Broken diamonds of lower intervals}
The final piece of the puzzle is to verify that shifted lower intervals satisfy the broken diamond condition.

\begin{lemma}\label{lem:key}
Let $W$ be a Coxeter group and let $w, x \in W$. If $\prec$ is an $N(x)$-initial reflection ordering, then the surviving chain of every broken diamond in $\swx$ is $\prec$-increasing.
\end{lemma}
\begin{proof}
Let $[y,z]$ be a broken diamond in $\swx$, with surviving chain $y \lessdot b \lessdot z$ and dead middle $m \in W\setminus \swx$, and assume for the sake of contradiction that the surviving chain is $\prec$-decreasing. In the ambient Bruhat diamond $[y,z]_W$, exactly one of the two maximal chains is increasing. Because the chain through $b$ is decreasing, the
chain through the dead middle $m$ is the increasing one. We name its two reflection labels:
\[
  t_i \;:=\; \lambda_R(y \lessdot m) \;=\; y^{-1}m
  \qquad \text{and} \qquad
  t_a \;:=\; \lambda_R(m \lessdot z) \;=\; m^{-1}z,
\]
so that $t_i \prec t_a$.

Next, we track how these elements behave under the shift by $x$. Since $m$ is the dead middle, it fails to land in our lower interval ($mx \not\le w$). The elements $y$ and $z$ belong to $\swx$, so $yx \le w$ and $zx \le w$. The pairs $\{y, m\}$ and $\{m, z\}$ each differ by a reflection, so the shifted pairs $\{yx, mx\}$ and $\{mx, zx\}$ are comparable in the usual Bruhat order. If $mx\leq zx$ or $mx\leq yx$, we would get $mx \leq w$ giving us a contradiction. Hence we must have $yx < mx$ and $zx < mx$.

We now apply our switching lemma (Lemma~\ref{lem:switch}) to both pairs. Looking at the lower pair $\{y, m\} = \{y, yt_i\}$, the ambient order has $y < m$ and the shifted order preserves this with $yx < mx$. Because the direction did not change, the label $t_i$ cannot belong to $N(x)$. Looking at the upper pair $\{m, z\} = \{m, mt_a\}$, the ambient order has $m < z$ but the shifted order reverses this with $mx > zx$. Because the direction flipped, the label $t_a$ must belong to $N(x)$. But $N(x)$ is an initial section of $\prec$ and $t_i \prec t_a \in N(x)$, so $t_i \in N(x)$, giving us a contradiction.
\end{proof}

\subsection{Proof of the main theorem}
With all our tools developed, we can now assemble the pieces and prove our primary result, which also implies Theorem~\ref{thm:main_intro}.

\begin{theorem}\label{thm:main}
Let $W$ be an arbitrary Coxeter group and $w, x \in W$. For every reflection ordering $\prec$ of $T$ having $N(x)$ as an initial section, the right labeling $\lambda_R$ is an EL-labeling of $\swx$.
\end{theorem}

\begin{proof}
By Corollary~\ref{prop:uniqueminmax}, the shifted lower interval $\swx$ possesses a unique minimum and a unique maximum. By Lemma~\ref{lem:ambientcovers}, the covers of $\swx$ are ambient in $W$, and Lemma~\ref{lem:key} ensures that the surviving chain of every broken diamond in $\swx$ is $\prec$-increasing. As all conditions are met, we can apply Proposition~\ref{prop:mainmid} directly to $\swx$ and conclude that the right reflection labeling $\lambda_R$ is an EL-labeling of $\swx$.
\end{proof}

\begin{corollary}\label{cor:consequences}
Let $W$ be a Coxeter group and $w, x \in W$. Then the shifted lower interval $\swx$ is EL-shellable. In particular, for every interval $[u,v]$ of $\swx$, the order complex of the open interval $(u,v)$ is shellable and Cohen-Macaulay, and is homotopy equivalent to a wedge of equidimensional spheres, one for each $\prec$-decreasing maximal chain of $[u,v]$ \cite[Section~2.7]{Bjorner-Brenti05}, \cite{BW82,BW83}. The M\"obius function of $\swx$ is given by
\[
  \mu(u,v) \;=\; (-1)^{\ell(v) - \ell(u)} \cdot
  \#\{\text{decreasing maximal chains of } [u,v]\} .
\]
\end{corollary}

\section{Further Questions and Final Thoughts}\label{section:conclusions}
 
Proposition~\ref{prop:mainmid} implies that a shifted middle interval
$\suvx$ with a unique minimum, unique maximum, and ambient covers is EL-shellable, provided we can find a reflection ordering that makes the surviving chain of every broken diamond increasing. However, the argument behind our key lemma (Lemma~\ref{lem:key}) explains why $N(x)$-initial orderings cannot succeed in this case. In particular, a dead middle $m$ can be obtained in two different ways.  Moreover, these two ways ask for opposite setups.
Consider a broken diamond $[y,z]$ whose dead chain carries the labels $t_i = y^{-1}m$ and $t_a = m^{-1}z$. In this case the surviving chain is increasing exactly when $t_a \prec t_i$.
 
If $m$ dies by failing the upper bound ($mx \not\le w_2$), the argument of Lemma~\ref{lem:key} forces $t_i\notin N(x)$ while $t_a\in N(x)$.  The demand that $t_a \prec t_i$ requires reflections in $N(x)$ to come first, which is granted whenever a reflection ordering is $N(x)$-initial. But if $m$ dies by failing the lower bound ($mx \not\ge w_1$), then everything flips.  Specifically, we have $t_i\in N(x)$ and
$t_a\notin N(x)$ and hence the reflection ordering needs to have $N(x)$ at the very end ($N(x)$-final).  As soon as $x$ is neither the identity nor the longest element (when $W$ is finite), a middle interval can have dead middles of both kinds at once.  In this case, no reflection ordering can be both $N(x)$-initial and $N(x)$-final at the same time.  
 
One can sometimes get around this issue by mixing, placing some parts of $N(x)$ early and others late, and in our computer experiments this often works, but not always.  There are middle intervals in which two different broken diamonds make conflicting demands on the exact same pair of reflections, so that no reflection ordering can satisfy the broken-diamond condition. 

Even worse, there exist shifted intervals which are graded posets but are not shellable (moreover, the associated order complex is not Cohen-Macaulay). To see this, consider the highlighted interval $[14325,53412]$ in Figure~\ref{fig:shellingcounter}, which is the Hasse diagram of the shifted interval $[13254,45132]\,(34152)^{-1}$. We have two maximal chains with disjoint interiors, meaning that the link of some face is not connected.

\begin{figure}[h]
 \begin{center}
  \vspace{3.7mm}
    \begin{tikzpicture}[scale=1.35,
  every node/.style={font=\small},
  elt/.style={inner sep=1.5pt},
  hi/.style={inner sep=1.5pt, text=blue},
  lab/.style={font=\scriptsize, fill=white, inner sep=0.8pt}]
  \node[elt] (a12345) at (0.00,0) {$12345$};
  \node[elt] (a12354) at (-2.22,1) {$12354$};
  \node[elt] (a12435) at (0.00,1) {$12435$};
  \node[elt] (a13245) at (2.17,1) {$13245$};
  \node[elt] (a12453) at (-2.93,2) {$12453$};
  \node[elt] (a13254) at (-1.45,2) {$13254$};
  \node[elt] (a14235) at (-0.01,2) {$14235$};
  \node[elt] (a13425) at (1.44,2) {$13425$};
  \node[elt] (a23145) at (2.91,2) {$23145$};
  \node[elt] (a13452) at (-3.31,3) {$13452$};
  \node[elt] (a14253) at (-2.19,3) {$14253$};
  \node[elt] (a23154) at (-0.98,3) {$23154$};
  \node[hi] (a14325) at (0.00,3) {$14325$};
  \node[elt] (a24135) at (1.12,3) {$24135$};
  \node[elt] (a32145) at (2.24,3) {$32145$};
  \node[elt] (a23415) at (3.29,3) {$23415$};
  \node[hi] (a14352) at (-3.12,4) {$14352$};
  \node[elt] (a24153) at (-1.81,4) {$24153$};
  \node[elt] (a32154) at (-0.60,4) {$32154$};
  \node[elt] (a42135) at (0.69,4) {$42135$};
  \node[hi] (a24315) at (1.89,4) {$24315$};
  \node[elt] (a32415) at (3.14,4) {$32415$};
  \node[hi] (a34152) at (-2.61,5) {$34152$};
  \node[elt] (a42153) at (-0.89,5) {$42153$};
  \node[elt] (a52134) at (0.83,5) {$52134$};
  \node[hi] (a42315) at (2.64,5) {$42315$};
  \node[hi] (a43152) at (-2.17,6) {$43152$};
  \node[elt] (a52143) at (-0.01,6) {$52143$};
  \node[hi] (a52314) at (2.21,6) {$52314$};
  \node[hi] (a53142) at (-1.48,7) {$53142$};
  \node[hi] (a52413) at (1.40,7) {$52413$};
  \node[elt] (a54132) at (-1.46,8) {$54132$};
  \node[hi] (a53412) at (1.44,8) {$53412$};
  \node[elt] (a54312) at (0.00,9) {$54312$};
  \draw (a12345) -- node[lab,pos=0.5] {$t_{45}$} (a12354);
  \draw (a12345) -- node[lab,pos=0.5] {$t_{34}$} (a12435);
  \draw (a12345) -- node[lab,pos=0.5] {$t_{23}$} (a13245);
  \draw (a12354) -- node[lab,pos=0.5] {$t_{35}$} (a12453);
  \draw (a12354) -- node[lab,pos=0.46] {$t_{23}$} (a13254);
  \draw (a12435) -- node[lab,pos=0.46] {$t_{45}$} (a12453);
  \draw (a12435) -- node[lab,pos=0.7] {$t_{24}$} (a13425);
  \draw (a12435) -- node[lab,pos=0.42] {$t_{23}$} (a14235);
  \draw (a13245) -- node[lab,pos=0.66] {$t_{45}$} (a13254);
  \draw (a13245) -- node[lab,pos=0.5] {$t_{34}$} (a13425);
  \draw (a13245) -- node[lab,pos=0.46] {$t_{24}$} (a14235);
  \draw (a13245) -- node[lab,pos=0.5] {$t_{13}$} (a23145);
  \draw (a12453) -- node[lab,pos=0.5] {$t_{25}$} (a13452);
  \draw (a12453) -- node[lab,pos=0.42] {$t_{23}$} (a14253);
  \draw (a13254) -- node[lab,pos=0.66] {$t_{35}$} (a13452);
  \draw (a13254) -- node[lab,pos=0.5] {$t_{25}$} (a14253);
  \draw (a13254) -- node[lab,pos=0.38] {$t_{13}$} (a23154);
  \draw (a13425) -- node[lab,pos=0.18] {$t_{45}$} (a13452);
  \draw (a13425) -- node[lab,pos=0.26] {$t_{23}$} (a14325);
  \draw (a13425) -- node[lab,pos=0.5] {$t_{14}$} (a23415);
  \draw (a14235) -- node[lab,pos=0.26] {$t_{45}$} (a14253);
  \draw (a14235) -- node[lab,pos=0.5] {$t_{34}$} (a14325);
  \draw (a14235) -- node[lab,pos=0.46] {$t_{13}$} (a24135);
  \draw (a23145) -- node[lab,pos=0.34] {$t_{45}$} (a23154);
  \draw (a23145) -- node[lab,pos=0.5] {$t_{34}$} (a23415);
  \draw (a23145) -- node[lab,pos=0.58] {$t_{24}$} (a24135);
  \draw (a23145) -- node[lab,pos=0.74] {$t_{12}$} (a32145);
  \draw (a13452) -- node[lab,pos=0.5] {$t_{23}$} (a14352);
  \draw (a14253) -- node[lab,pos=0.5] {$t_{35}$} (a14352);
  \draw (a14253) -- node[lab,pos=0.46] {$t_{13}$} (a24153);
  \draw[blue,line width=1pt,dashed] (a14325) -- node[lab,text=blue,pos=0.7] {$t_{45}$} (a14352);
  \draw[blue,line width=1pt,dashed] (a14325) -- node[lab,text=blue,pos=0.34] {$t_{14}$} (a24315);
  \draw (a23154) -- node[lab,pos=0.66] {$t_{25}$} (a24153);
  \draw (a23154) -- node[lab,pos=0.46] {$t_{12}$} (a32154);
  \draw (a23415) -- node[lab,pos=0.58] {$t_{23}$} (a24315);
  \draw (a23415) -- node[lab,pos=0.5] {$t_{12}$} (a32415);
  \draw (a24135) -- node[lab,pos=0.5] {$t_{45}$} (a24153);
  \draw (a24135) -- node[lab,pos=0.62] {$t_{34}$} (a24315);
  \draw (a24135) -- node[lab,pos=0.66] {$t_{12}$} (a42135);
  \draw (a32145) -- node[lab,pos=0.38] {$t_{45}$} (a32154);
  \draw (a32145) -- node[lab,pos=0.34] {$t_{34}$} (a32415);
  \draw (a32145) -- node[lab,pos=0.38] {$t_{14}$} (a42135);
  \draw[blue,line width=1pt,dashed] (a14352) -- node[lab,text=blue,pos=0.5] {$t_{13}$} (a34152);
  \draw (a24153) -- node[lab,pos=0.5] {$t_{15}$} (a34152);
  \draw (a24153) -- node[lab,pos=0.54] {$t_{12}$} (a42153);
  \draw[blue,line width=1pt,dashed] (a24315) -- node[lab,text=blue,pos=0.5] {$t_{12}$} (a42315);
  \draw (a32154) -- node[lab,pos=0.58] {$t_{25}$} (a34152);
  \draw (a32154) -- node[lab,pos=0.5] {$t_{15}$} (a42153);
  \draw (a32154) -- node[lab,pos=0.54] {$t_{14}$} (a52134);
  \draw (a32415) -- node[lab,pos=0.5] {$t_{13}$} (a42315);
  \draw (a42135) -- node[lab,pos=0.62] {$t_{45}$} (a42153);
  \draw (a42135) -- node[lab,pos=0.5] {$t_{34}$} (a42315);
  \draw (a42135) -- node[lab,pos=0.5] {$t_{15}$} (a52134);
  \draw[blue,line width=1pt,dashed] (a34152) -- node[lab,text=blue,pos=0.5] {$t_{12}$} (a43152);
  \draw (a42153) -- node[lab,pos=0.5] {$t_{25}$} (a43152);
  \draw (a42153) -- node[lab,pos=0.5] {$t_{14}$} (a52143);
  \draw[blue,line width=1pt,dashed] (a42315) -- node[lab,text=blue,pos=0.5] {$t_{15}$} (a52314);
  \draw (a52134) -- node[lab,pos=0.5] {$t_{45}$} (a52143);
  \draw (a52134) -- node[lab,pos=0.5] {$t_{34}$} (a52314);
  \draw[blue,line width=1pt,dashed] (a43152) -- node[lab,text=blue,pos=0.5] {$t_{14}$} (a53142);
  \draw (a52143) -- node[lab,pos=0.5] {$t_{34}$} (a52413);
  \draw (a52143) -- node[lab,pos=0.5] {$t_{25}$} (a53142);
  \draw[blue,line width=1pt,dashed] (a52314) -- node[lab,text=blue,pos=0.5] {$t_{35}$} (a52413);
  \draw[blue,line width=1pt,dashed] (a52413) -- node[lab,text=blue,pos=0.5] {$t_{25}$} (a53412);
  \draw[blue,line width=1pt,dashed] (a53142) -- node[lab,text=blue,pos=0.5] {$t_{34}$} (a53412);
  \draw (a53142) -- node[lab,pos=0.5] {$t_{24}$} (a54132);
  \draw (a53412) -- node[lab,pos=0.5] {$t_{23}$} (a54312);
  \draw (a54132) -- node[lab,pos=0.5] {$t_{34}$} (a54312);
\end{tikzpicture}
    \captionsetup{width=1.0\linewidth}
  \captionof{figure}{The interval $[13254,45132]$ shifted using $x=34152$.}
  \label{fig:shellingcounter}
 \end{center}
\end{figure}
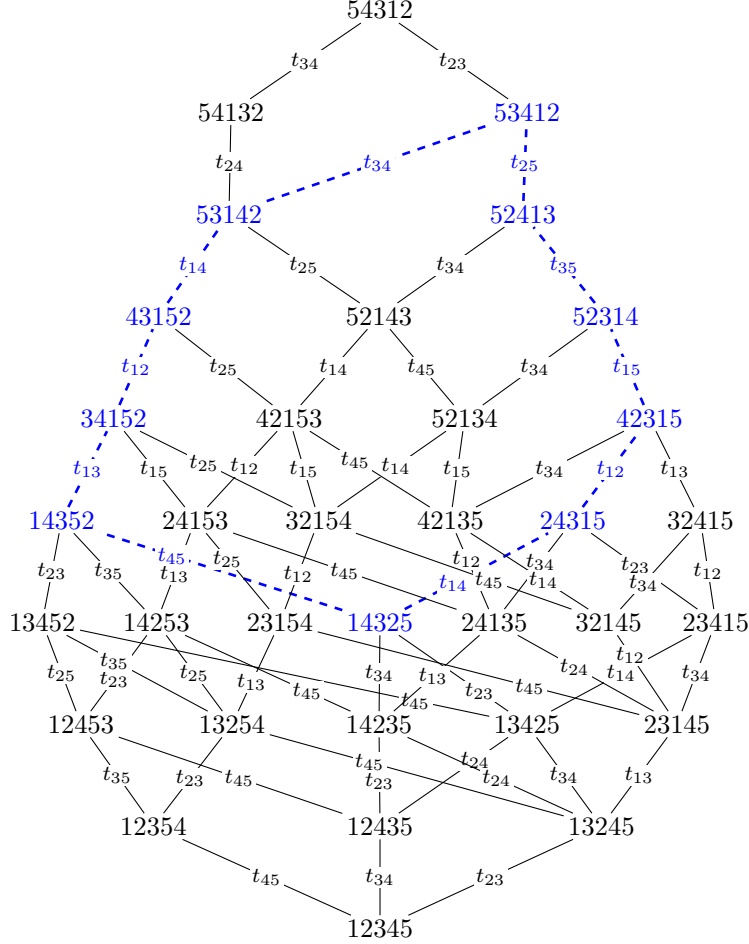

A property to consider for studying middle intervals would be whether they admit a \newword{recursive atom ordering}.  This criterion is equivalent to CL-shellability, which relaxes an EL-labeling by allowing the label of a cover to depend on the maximal chain leading up to it \cite{BW82,BW83}. CL-shellability still yields all the topological consequences of Corollary~\ref{cor:consequences}.
 
\begin{question}\label{q:middleCL}
Is there a way to describe when $\suvx$ is graded/CL-shellable?
\end{question}
 
Corollary~\ref{cor:consequences} shows that the M\"obius function of a shifted lower interval counts its decreasing chains up to sign. Unlike in Bruhat order, where every interval has exactly one decreasing chain, here the count need not equal $1$.  Shifted intervals can fail to be thin, so they need not be Eulerian. Already the rank-$2$ chain $[4213, 4321]$ in Example~\ref{ex:longstem} has M\"obius value $0$. It would be interesting to determine exactly which pairs $(w,x)$ produce
thin intervals, to map out the possible values of the M\"obius
function, and to develop analogues of the $R$-polynomials and Kazhdan-Lusztig polynomials, whose classical recursions rest on the rank-$2$ structure that broken diamonds disturb.
 
At its core, the combinatorics of shifted Bruhat intervals ties back to geometry. As mentioned in the introduction, we began studying shifted intervals because of an upcoming project \cite{LOR}, where shifted intervals index the cells of affine pavings of Richardson varieties cut out by shifted Borel subgroups, with the induced order governing the closure data. We expect the shellability established in Theorem~\ref{thm:main} to reflect the topology of these pavings, in the same way that shellability of Bruhat intervals does for Schubert varieties, with the decreasing-chain counts computing local invariants of the stratification. Making this geometric dictionary precise would be an interesting topic.

\subsection*{Acknowledgments}
This research was carried out primarily during the 2026 Honors Summer Math Camp at Texas State University. The authors appreciate the support from the camp and also thank Texas State University for providing support and a great working environment. E.R. was supported by a grant from the Simons Foundation 941273.

The authors used Claude (Anthropic) and Gemini (Google) for editing assistance and for writing computer code used to search for examples and to verify computational claims. All mathematical content and computational results were checked by the authors, who take full responsibility for the paper.

\printbibliography
\end{document}